\documentclass{amsart}
\usepackage{hyperref}

\newcommand{\RR}{\mathbb{R}}
\newcommand{\CC}{\mathbb{C}}
\newcommand{\front}[1]{\href{http://front.math.ucdavis.edu/#1}{arXiv:#1}}
\usepackage{amsmath}
\usepackage{amsthm}
\usepackage{amssymb}
\usepackage{amsfonts}
\usepackage{latexsym}
\usepackage{delarray}
\usepackage[dvips]{graphics}
\usepackage{epsfig}
\usepackage{color}

\newtheorem{prop}{Proposition}
\newtheorem{thm}[prop]{Theorem}

\newtheorem{lemma}[prop]{Lemma}

\newtheorem{theorem}[prop]{Theorem}
\newtheorem{proposition}[prop]{Proposition}
\newtheorem{corollary}[prop]{Corollary}

\theoremstyle{remark}
\newtheorem{remark}{Remark}

\title[Three notions of tropical rank for symmetric matrices]{Three notions of
tropical rank \\ for symmetric matrices}
\author{Dustin Cartwright \and Melody Chan}
\address{Dept.\ of Mathematics, University of California, Berkeley, CA 94720,
USA}
\email{dustin@math.berkeley.edu, mtchan@math.berkeley.edu}

\begin{document}
\begin{abstract}
We introduce and study three different notions of tropical rank for symmetric
and dissimilarity matrices in terms of minimal decompositions into rank~$1$
symmetric matrices, star tree matrices, and tree matrices. Our results provide a close
study of the tropical secant sets of certain nice tropical varieties, including
the tropical Grassmannian. In particular, we determine the dimension of each
secant set, the convex hull of the variety, and in most cases, the smallest
secant set which is equal to the convex hull.
\end{abstract}
\maketitle

\section{Introduction}
In this paper, we study tropical secant sets and rank for symmetric matrices.
Our setting is
the \emph{tropical semiring} $(\RR\cup\{ \infty\}, \oplus , \odot )$, where
tropical addition is given by $ x  \oplus y =\min(x,y)$  and tropical
multiplication is given by $x \odot  y = x+y$.
The $k$th  \emph{tropical secant set} of a subset $V $  of $\RR ^ N$  is defined to be the set of points
\begin{equation*}
\{x\in\RR ^ N: x=v_1 \oplus  \cdots  \oplus  v_k,\;v_i\in V\},
\end{equation*}
where $ \oplus $  denotes coordinate-wise minimum. This set is typically not a
tropical variety and thus we prefer the term ``secant set'' to ``secant
variety,'' which has appeared previously in the literature.  The \emph{rank} of
a point $ x\in\RR ^ N$  with respect to~$ V $  is the smallest
integer $ k$ such that $ x$ lies in the $ k $th tropical secant set of $ V $, or
$\infty $ if there is no
such $k$.

In~\cite{dss}, Develin, Santos, and Sturmfels define the Barvinok rank of a
matrix, not necessarily symmetric, to be the
rank with respect to the subset of $n \times n$ rank~$1$ matrices, and their
definition serves as a model for ours.  In addition, they define two other
notions of rank, Kapranov rank and tropical rank, for which there are no
analogues in this paper. Further examination of ranks of not necessarily
symmetric matrices can be found in the review article~\cite{akian}.

We give a careful examination  of secant sets and rank with respect to three
families of tropical varieties in the space of symmetric matrices and the space
of dissimilarity matrices.  By a
\emph{$n\times n$ dissimilarity matrix} we simply mean a function from
${[n]\choose 2}$ to $\RR$, which we will write as a symmetric matrix without any
entries on the diagonal. 
There is a natural projection from $n\times n$ symmetric matrices to $n
\times n$ dissimilarity matrices which we denote
by~$\pi$.  For example,
\begin{equation} \label{eqn:intro-exs}
M=
\begin{bmatrix}
0 & 1 & 0 & 0\\
1 & 0 & 0 & 0\\
0 & 0 & 0 & 1\\
0 & 0 & 1 & 0
\end{bmatrix} \quad\mbox{and}\quad\pi (M) =
\begin{bmatrix}
* & 1 & 0 & 0\\
1 & * & 0 & 0\\
0 & 0 & * & 1\\
0 & 0 & 1 & *
\end{bmatrix}
\end{equation}
are a symmetric matrix and dissimilarity matrix respectively.

Our first family is the \emph{tropical Veronese} of degree $2$, which is the
tropicalization of the classical space of symmetric matrices of rank~1. It is a
classical linear subspace of the space of symmetric matrices consisting of those
matrices which can be written as $v ^ T\odot v$ for some row vector $v$. The
rank of a matrix with respect to the tropical Veronese is called \emph{symmetric
Barvinok rank}, because it is the symmetric analogue of Barvinok
rank.

Second, we consider the \emph{space of star trees}, which is the image of the
tropical Veronese under the projection $\pi$.  Equivalently, it can be obtained
by first projecting the classical Veronese onto its off-diagonal
entries and then tropicalizing. The classical variety and its secant varieties
were studied in~\cite{drton}. The tropical variety is a classical linear
subspace of the space of dissimilarity matrices, and we call the corresponding
notion of rank \emph{star tree rank}. The name reflects the fact that the
matrices
with star tree rank~$1$ are precisely those points of the tropical Grassmannian
which correspond to trees with no internal edges, i.e.\ star trees, in in the
identification below.

Third, we consider the \emph{tropical Grassmannian} $G_{2,n}$, which is the
tropicalization of the Grassmannian of 2-dimensional subspaces in an
$n$-dimensional vector space, and was first studied in~\cite{ss}.  It consists
of exactly those dissimilarity
matrices arising as the distance matrix of a weighted tree with $n$ leaves
in which internal edges have negative weights. Therefore, we call the points in
the tropical Grassmannian \emph{tree matrices}, and call rank with respect to
the tropical Grassmannian the \emph{tree rank}. Note that our definition of tree
rank differs from that in~\cite[Ch.~3]{ps}, which uses a different notion of
mixtures.

Our first two families are examples of classical linear spaces, whose secant
sets were studied by Mike Develin~\cite{de}.  He defines a natural polyhedral
fan such that
the tropical secant set is the support of this polyhedral fan.  Moreover,
Theorem~2.1 in~\cite{de} gives an algorithm for computing the rank of a point
with respect to a fixed linear space.  In contrast, we do not know of a good
algorithm for computing the rank with
respect to the tropical Grassmannian (see Section~\ref{sec:open-questions}), and
we do not know of a natural fan structure.

\begin{figure}
\begin{center}
\input{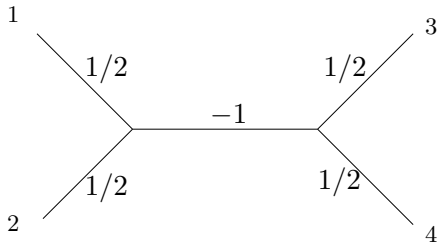}
\label{tree_rank}
\end{center}
\caption{Weighted tree whose distance matrix is $\pi(M)$
from~(\ref{eqn:intro-exs}).} \label{fig:example-tree}
\end{figure}
We use our examples of $M$ and $\pi(M)$ from~(\ref{eqn:intro-exs}) to illustrate
our three notions of rank.
Proposition~\ref{prop:zero-inf-rank} tells us that the symmetric Barvinok rank
of~$M$ is~4.
Theorem~\ref{thm:max-star-tree-rank} tells us that the star tree
rank of $ \pi(M)$ is~2. Explicitly, we have
\begin{equation*}
\pi (M)=\begin{bmatrix}
* & 1 & 0 & 0\\
1 & * & 2 & 2\\
0 & 2 & * & 1\\
0 & 2 & 1 & *
\end{bmatrix}\oplus\begin{bmatrix}
* & 1 & 2 & 2\\
1 & * & 0 & 0\\
2 & 0 & * & 1\\
2 & 0 & 1 & *
\end{bmatrix}.
\end{equation*}
Finally, the tree rank of $ \pi (M)$  is~1 by
Proposition~\ref{prop:tree-rank-01}, which can also be seen explicitly from the
weighted tree in Figure~\ref{fig:example-tree}.
This example shows that all three of our notions of rank can be different.

However, for any $n\times n$ symmetric matrix~$M$, we have
\begin{equation} \label{eqn:rank-inequals}
\operatorname{symmetric\ Barvinok\ rank}(M) \geq
\operatorname{star\ tree\ rank}(\pi (M)) \geq
\operatorname{tree\ rank}(\pi (M)).
\end{equation}
The first inequality follows from the fact that the set of dissimilarity
matrices of star tree rank~$1$ is the projection of the set of matrices of
symmetric Barvinok rank~$1$.
The second inequality follows from the fact that the space of star trees
is contained in the tropical Grassmannian.

The rest of the paper is organized as follows. In Section~\ref{sec:def-graph},
we present a technique for proving lower bounds on rank.  We introduce a
graph associated to a matrix for each of our notions of rank;
the chromatic number of this graph is a lower bound on the rank of the matrix.
The same technique applies to provide a lower bound to the rank of any
point with respect to any tropical prevariety, although in general it may
produce a hypergraph instead of a graph.

We examine symmetric Barvinok rank, star tree rank, and tree rank in Sections
\ref{sec:sym-barv-rank}, \ref{sec:star-tree-rank}, and~\ref{sec:tree-rank}
respectively.  We prove upper bounds on the rank in each case,
and with the exception of tree
rank, our upper bounds are sharp. We show that the
symmetric Barvinok rank of an $n\times n$ symmetric matrix can be infinite, but
even when the rank is finite it can exceed $n$, and in fact can grow
quadratically in~$n$ (Theorem~\ref{thm:max-rank}). For each notion of rank, the
set of matrices with
rank at most~$k$ is a union of polyhedral cones, and we compute the dimension of
these sets, defined as the dimension of the largest cone.
In each case, the
dimension of the tropical secant set equals the dimension of the clasical secant
variety, confirming Draisma's observation that tropical geometry provides 
useful lower bounds for the dimensions of classical secant varieties~\cite{d}.
Finally, we give a combinatorial characterization of each notion of rank for a $0/1$
matrix in terms of graph covers. 

In Section~\ref{sec:sym-barv-3}, we examine $3 \times 3$ symmetric matrices and
explicitly characterize the stratification by symmetric Barvinok rank.
In Sections \ref{sec:star-tree-5} and~\ref{sec:tree-5}, we do the same for the
$5 \times 5$ dissimilarity matrices and the stratifications by star tree rank
and tree rank respectively.
In particular, we show that the lower bounds from the chromatic number in
Section~\ref{sec:def-graph} are exact in these cases.
We close with some open problems in Section~\ref{sec:open-questions}.

\section{Lower bounds on rank via hypergraph coloring}\label{sec:def-graph}

Before we examine our three notions of rank, we give a general combinatorial
construction: a hypergraph whose chromatic number  yields a lower bound on rank. 

Recall that a  \emph{hypergraph} consists of a ground set, called vertices, and
a set of subsets of the ground set, called hyperedges. The \emph{chromatic
number} of a hypergraph~$H$, denoted $\chi(H) $, is the smallest number $ r$
such that
the vertices of $ H$  can be partitioned into $ r $ color classes with no
hyperedge of $ H$ monochromatic. In particular, if $H$
contains a hyperedge of size~$1$, then $\chi(H)$ is~$\infty$.

Now, suppose we have a tropical prevariety $
V\subseteq\RR ^ N$. Recall that a tropical polynomial
\begin{equation} \label{eqn:tropical-polynomial}
p(x_1, \ldots, x_N) = \bigoplus_{i=1}^t a_i \odot x_1^{c_{i1}} \odot\cdots\odot
x_N^{c_{iN}}
\end{equation}
defines a tropical hypersurface consisting of those vectors $x \in \RR^N$ such
that the minimum in evaluating $p(x)$ is achieved at least twice. A
\emph{tropical prevariety} is the intersection of finitely many tropical
hypersurfaces, and any
finite set~$S$ of tropical polynomials defining the prevariety $V$ is called a
\emph{tropical basis}.

Now, given a point $w \in \RR^N$ and a tropical basis $S$ for~$V$, we construct
a hypergraph on ground
set $ [N]$ as follows. Let $p$ from~(\ref{eqn:tropical-polynomial}) be a
tropical polynomial in~$S$, with all exponents $c_{ij} \geq 0$.
If the minimum is achieved uniquely when $p$ is evaluated at $w$, then we add a
hyperedge $E$ whose elements correspond to the coordinates that appear with
non-zero exponent in the unique minimal term.  The \emph{deficiency
hypergraph} of~$w$ with respect to~$V$ and~$S$ consists of hyperedges coming
from all polynomials in $S$ with a unique minimum at~$w$. In particular, the
deficiency hypergraph has no hyperedges (and thus has chromatic number $1$) if
and only if $w$ is in~$V$.

\begin{prop} \label{prop:deficiency-hypergraph}
If $H$ is the deficiency hypergraph constructed above, then
the rank of $w\in\RR ^ N$ with respect to $V\subseteq\RR ^ N$  is at least
$\chi(H)$. 
\end{prop}
\begin{proof}
Suppose that $w$ has rank $r$ with respect to~$V$, and let $w = v_1 \oplus
\cdots \oplus v_r$ be an
expression of $w$ as the tropical sum of $r$ points $v_i \in V$. We will
construct an $r$-coloring of the deficiency hypegraph~$H$, which will show that
$\chi(H) \leq r$. For each $i \in [N]$, there is at least one $v_j$ which agrees
with $w$ in the $i$th coordinate, so arbitrarily pick one such $j$ as the color
for vertex~$i$ in~$H$. Let $E$ be a hyperedge of~$H$ and $p$ the
associated tropical polynomial. We claim that $E$ cannot be monochromatic with
color~$j$. Each coordinate of $v_j$ is greater than or equal to the
corresponding coordinate of~$w$, so each term of $p(v_j)$ is greater than or
equal to the corresponding term of $p(w)$. On the other hand, the minimum in the
evaluation
of $p(v_j)$ is achieved at least twice, so the minimum must be strictly greater
than $p(w)$. Thus, $v_j$ cannot agree with $w$ for all coordinates in $E$, so
$E$ is not monochromatic of color~$j$. This holds for any color~$j$, so we have
constructed an $r$-coloring, and thus, $\chi(H) \leq r$.
\end{proof}

\begin{corollary}\label{cor:singleton}
If the deficiency hypergraph $H$ has a hyperedge of
size~$1$, then the rank of $w$ with respect to $V$  is infinite.
\end{corollary}

\begin{proof}
If $H$ has a hyperedge of size~$1$, then $\chi(H)$ is $\infty$, and thus the
rank of $w$ is infinite.
\end{proof}

We do not know of any examples in which this lower bound is
actually strict; see Section~\ref{sec:open-questions}.

For the varieties considered in this paper, we will take quadratic tropical
bases and thus the deficiency hypergraph will always be a graph (possibly with
loops). Accordingly, we will call it the \emph{deficiency graph}.

\section{Symmetric Barvinok rank}\label{sec:sym-barv-rank}

Recall from the introduction that the symmetric Barvinok rank of a symmetric
matrix $M$ is the
smallest number~$r$ such that $M$ can be written as the sum of $r$ rank~$1$
symmetric matrices. The $2\times 2$ minors $x_{ij} x_{kl} \oplus x_{il}x_{kj}$
of $M$ for $i \neq k$ and $l \neq j$ form a tropical basis for the variety of
rank~$1$ symmetric matrices. We will always construct our deficiency graph
with respect to this tropical basis.

Our first observation is that the symmetric Barvinok rank of a matrix can be
infinite. More precisely,
\begin{prop} \label{prop:inf-rank}
If $M$ is a symmetric matrix and $2M_{ij} < M_{ii} + M_{jj}$ for some $i$
and~$j$, then the symmetric Barvinok
rank of $M$ is infinite.
\end{prop}

\begin{proof}
The tropical polynomial $x_{ij}^2 \oplus x_{ii} x_{jj}$ is in the tropical
basis, so if $2M_{ij} <M_{ii} + M_{jj}$ for some $ i$  and $ j$, then the
deficiency graph for~$M$ has a loop at the node $ij$.
 Therefore, $M$
has infinite rank by Corollary~\ref{cor:singleton}.
\end{proof}

In fact, the converse to Proposition~\ref{prop:inf-rank} is also true; see
Theorem~\ref{thm:max-rank}.
In order to construct decompositions into rank~$1$ matrices, we need the
following lemma.
\begin{lemma} \label{lem:rank-1-extn}
Let $M$ be an $m \times m$ symmetric Barvinok rank~$1$ matrix, $n > m$ an
integer and $C$
any real number. Then there exists an $n \times n$ symmetric rank~$1$ matrix~$N$
such that the upper left $m \times m$ submatrix is $M$ and every other entry is
at least~$C$.
\end{lemma}

\begin{proof}
Since $M$ has rank~$1$, then $M = v^T \odot v$ for some row vector $v$. Let $C'
= \max\{\frac{1}{2}C, C- v_i\}$, and let $w$ be the vector consisting of $v$ followed by $C'$
repeated $n-m$ times. Then $N = w^T \odot w$ has the desired properties.
\end{proof}

\begin{remark} \label{rmk:rank-1-extn}
We will use the symbol $\infty$ in an entire row and column of a matrix to
denote sufficiently large values that maintain the property of being rank~$1$.
So, if $M$ is an $m \times m$ rank~$1$ matrix, then
\begin{equation*}
\begin{bmatrix}
M & \infty \\ \infty & \infty \end{bmatrix}
\end{equation*}
denotes the $(m+1)\times(m+1)$ matrix obtained by applying 
Lemma~\ref{lem:rank-1-extn} with $n = m+1$. The value of~$C$ will be clear from
the context.
\end{remark}

Next, we give a graph-theoretic characterization of the symmetric Barvinok rank
of $0/1$-matrices. We define a \emph{clique cover} of a simple graph~$G$ to be a
collection of $r$ complete subgraphs such that every edge and every vertex of $
G$  is in some element of the collection.  Given an $n \times n$ symmetric $0/1$
matrix  $M$ with zeroes on the diagonal, define $G_M$ to be the graph whose
vertices are the integers $[n]$, and which has an edge between $i$ and~$j$ if
and only if $M_{ij} = 0$.

\begin{prop} \label{prop:zero-inf-rank}
Suppose $M$ is a symmetric $0/1$ matrix with zeroes on the diagonal. Then the
symmetric Barvinok rank of $M$ is the size of a smallest clique cover
of~$G_{M}$.

On the other hand, suppose that $M$ is a symmetric $0/1$ matrix with at least
one entry of $1$ on the diagonal. If there exist $i$ and $j$ such that $M_{ii}
= 1$ and $M_{ij} = 0$, then the symmetric Barvinok rank of~$M$ is infinite.
Otherwise, let $M'$ be the maximal principal submatrix with zeroes on the
diagonal. The symmetric Barvinok rank of $M$ is one greater than the symmetric
Barvionk rank of $M'$.
\end{prop}

\begin{proof}
First suppose that all diagonal entries of $ M $ are $0$.
Let $G_1, \ldots, G_r$ be a clique cover of $G_M$.
Let $v_i$ be the  $ 0/1$ row vector whose $j$th entry is $0$ if $j$ is a vertex
of~$G_i$, and let $ M_i =v_i^T \odot v_i $.  Then we claim that $M = \bigoplus
M_i$.
 Each~$0$ entry in $M$
corresponds to an edge or vertex of $G_M$ which is in some $G_i$, so there
is a $0$ in $M_i$.  For $j \neq k$ such that $M_{jk} = 1$, some $G_i$
contains the vertex~$j$, and hence misses the vertex $k$, so $(M_i)_{jk} = 1$. On the other hand, the entry is at least $1$ in
the other rank~$1$ matrices because none of the corresponding graphs contains $ jk$  as an edge.  Thus the
rank of $M$ is at most~$r$.

Conversely, suppose $N$ is a rank~$1$ symmetric matrix in a decomposition of
$M$.  Since all entries of $M$ are non-negative, so are the entries of $N$, so
the rank~$1$ condition says $N_{ii} = N_{jj} = 0$ if and only if $N_{ij} = 0$.
By the ``if'' direction, we can define a graph $G_N$ whose vertices and edges
correspond to the diagonal and off-diagonal zeroes of $N$ respectively.  By the
``only if'' direction, this is a complete graph.  Every position with a $0$ in
$M$ must be~$0$ for some rank~$1$ matrix in the decomposition, so the graphs
$G_N$ form a clique cover of $G_M$ as $N$ ranges over all rank~$1$ symmetric
matrices in the decomposition.  Thus, the rank of $M$ is exactly~$r$.

Next, we suppose $ M$ has at least one $1$ on the diagonal and let $M'\subsetneq
M$ be as in the statement. If there exist $ i$  and $ j$  such that $
M_{ii} =1$
but $ M_{ij}=0 $, then the rank of $ M$  is infinite by
Proposition~\ref{prop:inf-rank}. Otherwise, we claim
that the rank of  $M$ is $ r+1 $. Extending each rank 1 summand of a minimal
decomposition of $ M'$  by Lemma~\ref{lem:rank-1-extn} and adding in the all ones matrix shows
that the rank of $M$  is at most $ r+ 1 $. On the other hand, it is
straightforward to check that a rank 1 summand containing a 1 entry on the
diagonal can contain no zeroes, so does not contribute to a clique cover for $
G_{ M'} $. So the rank of $M$  is exactly $ r+1 $ in this case.
\end{proof}

\begin{remark} \label{rmk:bipartite_graph}
This characterization gives us two families of matrices which have rank~$n$ and $\lfloor
n^2/4 \rfloor$ respectively, namely those corresponding to the trivial graph with $n$ isolated vertices
and the complete bipartite graph $K_{\lfloor n/2\rfloor, \lceil n/2\rceil}$.
In the latter case, $K_{\lfloor n/2\rfloor, \lceil n/2\rceil}$ is triangle-free,
so no clique can consist of more than one edge.  On the
other hand, there are $\lfloor n^2/4 \rfloor$ edges in the graph, so $\lfloor
n^2/4 \rfloor$ cliques are needed in a cover.
In fact, these two examples have the maximum possible rank for $n\times n$
matrices, as shown below.
\end{remark}

\begin{thm} \label{thm:max-rank}
Suppose that $M$ is a symmetric $n\times n$ matrix with $M_{ii} + M_{jj} \leq
2M_{ij}$ for all $i$ and~$j$.  Then the symmetric  Barvinok rank of $M$ is at
most $\max\{n, \lfloor n^2/4 \rfloor\}$, 
and this bound is tight.
Thus, every matrix with finite rank has rank at most $\max\{n, \lfloor
n^2/4\rfloor\}$.
\end{thm}

\begin{proof}
Subtracting $M_{ii}/2$ from
the $i$th row and column for all $i$ does not change the rank, so we can assume
that the diagonal entries of
$M$ are $0$ and hence, by hypothesis, the off-diagonal
entries are non-negative.

The statement  is trivial for $n=1$.
For $n=2$, we have 
\begin{equation} \label{eqn:decomp-2}
M = 
\begin{bmatrix}
0 & M_{12} \\
M_{12} & 2 M_{12}
\end{bmatrix} \oplus \begin{bmatrix}
2M_{12} & M_{12} \\
M_{12} & 0
\end{bmatrix}.
\end{equation}
For $n=3$, we assume, without loss of generality, that $M_{12} \geq M_{23}$.
Then
\begin{equation}
M =
\begin{bmatrix}
0 & M_{12} & M_{13} \\
M_{12} & 2M_{12} & M_{12} \!+\! M_{13} \\
M_{13} & M_{12}\! + \!M_{13} & 2 M_{13}
\end{bmatrix}
\oplus
\begin{bmatrix}
\infty & \infty & \infty \\
\infty & 0 & M_{23} \\
\infty & M_{23} & 2M_{23}
\end{bmatrix}
\oplus
\begin{bmatrix}
\infty & \infty & \infty \\
\infty & \infty & \infty \\
\infty & \infty & 0
\end{bmatrix}, \label{eqn:decomp-3}
\end{equation}
where ``$\infty$'' is as in Remark~\ref{rmk:rank-1-extn}.
For $n$ at least $4$, the proof is by induction, using
the following lemma:

\begin{lemma} \label{lem:max-rank-induction}
Let $n \geq 4$.  Suppose that for any $(n-2) \times (n-2)$ matrix $N$ of finite
rank, there exists a matrix $N'$ of rank at most $\lfloor (n-2)^2/4\rfloor$,
such that
$N'$ is identical to~$N$ except possibly in one diagonal entry, where the
entry of $N'$ is greater than or equal to the entry of~$N$. Then any $n\times
n$ matrix has rank at most~$\lfloor n^2/4\rfloor$.
\end{lemma}

\begin{remark}
The exceptional diagonal entry in the hypothesis makes the statement of the
lemma slightly stronger than what is required for an inductive proof of an
$\lfloor n^2/4 \rfloor$ upper bound. However, we will
use the lemma to establish the base cases of the induction, $n=4$ and $n=5$,
which require a weaker hypothesis because for $n=2$ and $n=3$ the $\lfloor n^2/4
\rfloor$ upper bound on rank does not hold.
\end{remark}

\begin{proof}[Proof of Lemma~\ref{lem:max-rank-induction}]
Let $M$ be an $n\times n$ matrix with finite rank.
Without loss of
generality, assume that the entry $M_{12}$ is minimal among all off-diagonal
elements. We apply the hypothesis to the principal submatrix indexed by $\{3,
\ldots, n\}$. 
Applying Lemma~\ref{lem:rank-1-extn}, we have a collection of at most $\lfloor
(n-2)^2/4\rfloor$ rank $1$ matrices whose tropical sum agrees with $M$ except
for the first two rows and columns and possibly one diagonal entry, which,
without loss of generality, we assume to be $M_{44}$. For each $4 \leq i \leq
n$, take the rank~$1$ matrix which has arbitrary large values except for the
$\{1,2,i\}$ principal submatrix, which is:
\begin{equation*}
\begin{bmatrix}
2M_{1i} & M_{1i} + M_{2i} & M_{1i} \\
M_{1i} + M_{2i} & 2M_{2i} & M_{2i} \\
M_{1i} & M_{2i} & 0
\end{bmatrix}.
\end{equation*}
Note that since $M_{12}$ was chosen to be minimal, we have that $M_{12} \leq
M_{1i} + M_{2i}$.

Finally, switching indices $1$ and $2$ if necessary, we can assume
that $M_{13} \geq M_{23}$, and hence $M_{12}+M_{13} \geq M_{23}$. Then, we take
two matrices which are ``$\infty$'' outside
of the $\{1,2,3\}$ principal matrices, which are, respectively,
\begin{equation*}
\begin{bmatrix}
0 & M_{12} & M_{13}\\
M_{12} & 2M_{12} & M_{12} + M_{13} \\
M_{13} & M_{12} + M_{13} & 2M_{13}
\end{bmatrix}
\quad \mbox{and} \quad
\begin{bmatrix}
\infty & \infty & \infty \\
\infty & 0 & M_{23} \\
\infty & M_{23} & 2M_{23} 
\end{bmatrix},
\end{equation*}
recalling the meaning of ``$\infty$'' from Remark~\ref{rmk:rank-1-extn}.

This yields a decomposition of $M$ into at most $\lfloor (n-2)^2/4
\rfloor + (n-3) + 2 = \lfloor n^2/4 \rfloor$ symmetric rank~$1$ matrices.
\end{proof}

To complete the proof of the Theorem~\ref{thm:max-rank}, we note that taking all
but the last term
of~(\ref{eqn:decomp-2}) and~(\ref{eqn:decomp-3}) gives the hypothesis to
Lemma~\ref{lem:max-rank-induction} for $n=4$ and~$5$. The desired upper bound
follows by induction, which consists of applying
Lemma~\ref{lem:max-rank-induction} with $N' = N$. 
Finally, Remark~\ref{rmk:bipartite_graph} shows that this bound is tight.
\end{proof}

\begin{theorem}
The dimension of the space of symmetric $n \times n$ matrices of symmetric
Barvinok rank at
most~$r$ is ${n+1 \choose 2} - {n-r+1 \choose 2}$, which is the dimension of the
classical secant variety, i.e.\ the space of classical symmetric matrices of
classical rank at most~$r$.
\end{theorem}

\begin{proof}
Let $D = {n+1 \choose 2} - {n -r+1 \choose 2}$.
The tropical secant variety is contained in the tropicalization of the classical
secant variety, so the dimension is at most $D$, by the Bieri-Groves
Theorem~\cite[Thm.\ A]{bg}.
Thus, it is sufficient to find an open neighborhood in which the tropical variety has
dimension~$D$. For $i$ from~$1$ to~$r$, let $v_i = (C, \ldots, C, v_{i,i},
\ldots, v_{i,n})$ be a vector with $C$ for the first $i-1$ entries. Choose
the coordinates $v_{i+1,j}$ to be smaller than all the $v_{i,j}$ and let
$C$ be very large. Then,
\begin{equation*}
v_1^T \odot v_1 \oplus \cdots \oplus v_r^T \odot v_r =
\begin{bmatrix}
2v_{11} & v_{11} + v_{12} & \cdots & v_{11} + v_{1n} \\
v_{11} + v_{12} & 2v_{22} & \cdots & v_{22} + v_{2n} \\
\vdots & \vdots & & \vdots \\
v_{11} + v_{1n} & v_{22} + v_{2n} & \ldots  & 2v_{rn}
\end{bmatrix}
\end{equation*}
This matrix is an injective function of the vector entries $v_{ij}$ for $i \leq
r$ and $j \geq i$. Thus, it defines a neighborhood of the $r$th secant set of
the desired dimension
\begin{equation*}
n + (n-1) + \cdots + (n-r+1) = {n+1 \choose 2} - {n-r+1
\choose 2} = D. \qedhere
\end{equation*}
\end{proof}

\section{Star tree rank}\label{sec:star-tree-rank}

Recall from the introduction that a star tree matrix is one which can be written
as $\pi(v^T \odot v)$ for $v \in \RR^n$ a row vector. The star tree matrices
from  a classical
linear space in the space of $n \times n$ dissimilarity matrices defined by the
tropical polynomials
\begin{equation} \label{eqn:basis-star-tree}
x_{ij} x_{kl} \oplus x_{ik} x_{jl}
\quad\quad\mbox{for $i$, $j$,
$k$, and $l$ distinct integers}.
\end{equation}
In this section, the deficiency graph will
always be taken with respect to this tropical basis.
The following lemma is an immediate consequence of Lemma~\ref{lem:rank-1-extn}.
\begin{lemma} \label{lem:star-tree-extn}
Let $M$ be an $m\times m$ star tree matrix, and let $n > m$ be
any integer and $C$ any real number. Then there exists an $n \times n$ star tree
matrix with $M$ as the upper left $m \times m$ submatrix and all other entries
greater than~$C$.
\end{lemma}

Unlike the case of symmetric Barvinok rank, the star tree rank is always
finite.
\begin{thm} \label{thm:max-star-tree-rank}
For $n$ at least $3$, the star tree rank of a $n \times n$ dissimilarity
matrix~$M$ is at most $n-2$, and this bound is sharp. In particular, the
dissimilarity matrix defined
by $M_{ij} = \min\{i,j\}$ has star tree rank $n-2$.
\end{thm}

\begin{proof}
The proof of the upper bound is by induction on $n$. For $n=3$, the equations
in~(\ref{eqn:basis-star-tree}) are trivial, and thus every $3\times 3$
dissimilarity matrix is a star
tree matrix.  For $n>3$, let $M$
be an $n \times n$ dissimilarity matrix, and denote by $M'$ the upper left
$(n-1) \times (n-1)$ submatrix. By the inductive hypothesis, we can write
$M'$ as the tropical sum of $n-3$ star tree matrices. We can
extend each of these to $n\times n$ star tree matrices by
Lemma~\ref{lem:star-tree-extn}, and their tropical sum will
agree with $M$ except in the last column and row. Let $w$ be the vector defined by $w_{i} = M_{in} + C$
for $i < n$ and $w_n = -C$, for $C$ a sufficiently large number. Then the tropical sum of
the previous $n-3$ matrices together with $\pi(w^ T\odot w)$ equals~$M$.

Now let $M$ be the dissimilarity matrix defined by $M_{ij} = \min\{i,j\}$ as in
the statement. We claim that the deficiency graph of $M$ has
chromatic number~$n-2$.  For every $i < j < k < l$, we have 
\begin{equation} \label{eqn:m-4-point}
M_{ik} + M_{jl} = M_{il} + M_{jk} = i + j < M_{ij} + M_{kl} = i + k,
\end{equation}
so the deficiency graph has an edge between $ik$ and $jl$, and an edge between
$il$ and $jk$. We refer to these types of edges as ``overlapping'' and
``nesting'' respectively.

We prove that the deficiency graph of $M$ has chromatic number at least $n-2$ by
induction on $n$.  The case of $n=3$ is clear. Let $n$ be greater than $3$ and
fix a coloring of the deficiency graph of $M$. Let $S$ be  the set of nodes of
the same color~$c$ as the node~$1n$. There is a ``nesting'' edge between $1n$
and every node other than those of the form $1i$ or $in$ for some~$i$.  Thus,
every node in $S$ is either of the form $1i$ or $in$.  Furthermore, because of
the ``overlapping'' edges, there must be an integer $m$ such that if $1i$ is in
$S$, then $i\leq m$ and if $jn$ is in $S$, then $m \leq j$. Now consider the set
of $n-3$ nodes consisting of $1i$ for $m < i < n$ and $jn$ for $1 < j < m$. By
our construction of~$m$, none of them is in~$S$. Therefore, if they have
distinct colors, then the coloring has at least $n-2$ colors, which is what we
wanted to show.

Otherwise, two of these nodes have the same color, and by the
``overlapping'' edges and symmetry we can assume that they are $1i$ and $1i'$,
with $m < i < i'$, which have color~$c' \neq c$.  Any node $jk$ with $j < k$ and
$1 < j < i$
will share an edge with one of these two nodes, so cannot have color $c'$.  On
the other hand, if the node $1l$ with $l \leq m$ is adjacent to 
$jk$ with $j <k$, then we must have $ 1 < j < l$. But $l \leq m < i$, so $jk$
cannot have color~$c'$. Thus, we can assume that all of the
nodes $1l$ with $l \leq m$ also have color~$c'$ without changing the fact that
the coloring is
proper. With this change, the only nodes with color~$c$ are of the form~$jn$. By
restricting to the nodes with coordinates less than~$n$, we have a coloring of
the deficiency graph of the $(n-1)\times (n-1)$ matrix without the color~$c$, so
by the inductive hypothesis we're done.
\end{proof}

\begin{remark}
The matrix with maximal star tree rank in the previous theorem is in fact in the
Grassmannian, i.e.\ it has tree rank~$1$. Indeed, from~(\ref{eqn:m-4-point}), we
see that the four-point condition holds.  Alternatively, $M$ arises as the
distance matrix of the following weighted tree.
Let $T$ be the caterpillar tree with $n$ internal vertices, connected in
order by edges of weight $-1/2$. The $i$th leaf vertex is connected to the $i$th
internal vertex by an edge of weight $i/2$. For $i < j$, the distance from
leaf~$i$ to leaf~$j$ is $i/2 + (j-i)(-1/2) + j/2 = i$, which is equal to the
corresponding entry in the matrix $M$ from Theorem~\ref{thm:max-star-tree-rank}.
In order to make a proper phylogenetic tree, we should remove the first and last
internal vertices and combine the adjacent edge weights.
\end{remark}

Next, we give a graph theoretic
characterization of the star tree rank of $0/1$-matrices.  For $M$ a
$0/1$ dissimilarity matrix, we define $G_M$ to be the graph whose edges
correspond to the zeroes of $M$. As in the case of symmetric Barvinok rank,
we can characterize the star tree rank of~$M$ in terms of covers of~$G_M$, this
time by
both cliques and star trees. We will also say that a cover
of~$G_M$ by cliques and star trees is a \emph{solid cover} if for every pair of
distinct vertices $i$ and $j$ either:
\begin{enumerate}
\item there is an edge between $i$ and $j$,
\item either $i$ or $j$ belongs to a clique in the cover,
\item either $i$ or $j$ is the center of a star tree in the cover, or
\item both $i$ and $j$ are leaves of the same star tree.
\end{enumerate}

\begin{prop} \label{prop:star-tree-rank-01}
Let $M$ be a $0/1$ dissimilarity matrix.  Let $r$ be the minimal number of
graphs in a cover of $G_M$ by cliques and star trees, such that every edge (but
not necessarily every vertex) is in some element of the cover. Then $M$ has star
tree rank either $r$ or~$r+1$.

Moreover, if $G_M$ has a solid cover by $r$ graphs, then $M$ has star tree
rank~$r$.
\end{prop}

\begin{proof}
Let $G_1, \ldots, G_r$ be a cover of $G_M$ by cliques and star trees. For $G_i$
a clique, define $v_i$ to be the $0/1$ vector whose $0$ entries correspond to
the vertices of the clique.  If $G_i$ is a star tree consisting of a central
vertex $c_i$ and edges to vertices in the set $I_i$, then define $v_i$ to be the
row vector which is $-1/2$ in the $c_i$ entry, $1/2$ for the entries
corresponding to $I_i$ and $3/2$ otherwise.  In either case, define $M_i =
\pi(v_i^ T \odot v_i)$. Then $M_i$ has $0$ entries corresponding to the edges of
$G_i$. Thus, the tropical sum of the $M_i$ has the same $0$ entries as $M$.
Moreover, if the cover is a solid cover then the tropical sum is equal
to $M$, so $M$ has rank at most~$r$. Otherwise, some of the positive entries of
the tropical sum are greater than $1$. By
additionally taking the
tropical sum with the all ones matrix, we see that $M$ has rank at most $r+1$.

Conversely, suppose that $M_i = \pi(v_i^ T \odot v_i)$ is a term in
a representation of~$M$ as the tropical sum of star tree matrices. Then we claim
that the zeroes of $M_i$ correspond to
either a star tree or a complete graph. If all entries of $v_i$ are
non-negative, then the zeroes of $M_i$ correspond to the complete graph on the
vertices where $v_i$ is $0$. Otherwise, since $M_i$ must be non-negative, there
can be at most one negative entry of~$v_i$, say with value $-a$; then all
other entries must be at least $a$. Then the $0$ entries of $M_i$ correspond to
the star tree with edges between the entry with $-a$ and the entries with
value~$a$. Thus, any decomposition of $M$ as the tropical sum of star tree
metrics yields a cover of $G_M$ by cliques and star trees, so $M$ has tropical
rank at least $r$.
\end{proof}


We do not know if the definition of a solid cover can be weakened in any way. In
other words, we do not know of any $0/1$ matrices $M$ such that $r$ is the
minimal size of a cover of $G_M$ by cliques and star trees, and $G_M$ does not
have a solid cover of size $r$, but $M$ has star tree rank~$r$.

In contrast to symmetric Barvinok rank, the upper bound of $n-2$ on the star
tree rank of an $ n\times n$  dissimilarity matrix cannot be achieved by a
$0/1$ matrix for large~$n$.  Recall that the \emph{Ramsey
number} $R(k,k)$ is the smallest integer such that any graph on at least
$R(k,k)$
vertices has either a clique or a independent set of size~$k$. Then we have the
following stronger bound on the star tree rank of a $0/1$ matrix.

\begin{proposition}
For $n \geq R(k,k)$, any $n\times n$
$0/1$ dissimilarity matrix has star tree rank at most $n - k +1$.
\end{proposition}
\begin{proof}
By the assumption on $n$, the graph $G_M$ has either a clique of size $k$ or an
independent set of size $k$.
In the former case, we can cover
$G_M$ by a star tree centered at each vertex not part of the clique, together
with the clique itself. This gives a solid cover by $n-k+1$ subgraphs.
In the latter
case, we can just take the star trees centered
at the vertices not in the independent set, giving a cover of
$G_M$ by $n-k$ subgraphs. In either case,
Proposition~\ref{prop:star-tree-rank-01} shows that $M$
has rank at most $n-k+1$.
\end{proof}

\begin{corollary}
For $n \geq 18$, every $n \times n$ $0/1$ dissimilarity matrix has star tree
rank at most $n-3$.
\end{corollary}
\begin{proof}
The Ramsey number $R(4,4)$ is $18$~\cite{r}.
\end{proof}

\begin{theorem}
Let $ r$  and $ n$  be positive integers. Then the dimension of the space of
dissimilarity $ n\times n$  matrices of star tree rank at most $ r$  is
\begin{equation*}
\min\left\{{ n+1\choose 2} -{ n-r+1\choose 2},~{ n\choose 2}\right\}.
\end{equation*}
\end{theorem}

\begin{proof}
Let $D = \min\left\{{n+1 \choose 2} - {n-r+1 \choose 2}, {n \choose 2}\right\}$
be the dimension from the theorem statement. The dimension cannot be any larger
than $D$ by the Bieri-Groves Theorem~\cite[Thm.\ A]{bg}, because $D$ is the
dimension of the classical secant variety, according to Theorem~2
in~\cite{drton}.
Therefore, it is sufficient to construct a matrix with star tree rank~$r$ which
has a $D$-dimensional neighborhood of star tree rank~$r$ matrices. If $r \geq
n$, then $D$ is ${n \choose 2}$, the dimension of the set of $n\times n $
dissimilarity matrices. Since higher secant sets cannot have smaller dimension,
it is sufficient to assume $r \leq n$.

We will construct $r$ vectors $v_1, \ldots, v_r$, with $v_{k,i}$ denoting the
$i$th entry of $v_k$, and then define $M$ to be the tropical sum of the star
tree matrices $\pi(v_k^T \odot v_k)$.
First, we fix any order on the set of pairs of distinct integers $S = \{(i,j) :
r < i < j \leq n\}$. Then, for $1 \leq k \leq r$ and $r < i \leq n$, we choose
$v_{k,i}$ as
follows: if the $k$th pair of integers includes $i$, then we choose $v_{k,i}$
in the range $0<v_{k,i} < 1$ and otherwise we choose $v_{k,i} > 2$.
For $i$ in the range $k \leq i \leq r$, we choose $v_{k,i}$ inductively
beginning with $k=r$.  We choose $v_{r, r}$ arbitrarily.
Then for $k < r$, and $k \leq i \leq r$, we choose  $v_{k, i}$ to be much
greater than any of the $v_{k+1,j}$ already chosen. Finally, let $C$
be a large real number and set $v_{k,
i}$ equal to $C$, for $i < k$. Let $M$ be the tropical sum of the $M_k :=
\pi(v_k^T \odot v_k)$.

We claim that the set of matrices which can be gotten in this way forms a
$D$-dimensional affine linear neighborhood.  For $i \leq r$ and $i < j$, the
$(i,j)$ entry of $M$ comes from $M_i$ and in particular, is equal to $v_{i,i} +
v_{i,j}$. For a fixed $i$, and taking $j > i$, these entries give us $n-i$
linearly independent functions on the matrix~$M$. Moreover, if $(i,j)$ is the
$k$th pair in the ordering on $S$, and $k \leq r$, then the $(i,j)$ entry of $M$
comes from $M_k$ and is equal to $v_{k,i} + v_{k,j}$. These are linearly
independent from each other and from all of the previous functions. Since the
size of~$S$ is ${n-r \choose 2}$, the number of linearly independent functions
on the matrix $M$ is
\begin{align*}
(n-1) + (n-2) + \cdots + (n-r) + &\min\left\{r, {n -r \choose 2}\right\} \\
&= {n \choose 2} - {n - r \choose 2} + \min\left\{r, {n-r \choose
2}\right\} \\
&= \min\left\{{n+1 \choose 2} - {n-r+1\choose 2}, {n\choose 2}\right\},
\end{align*}
which is the desired dimension, $D$.
\end{proof}

\begin{remark}
In fact, the difficult part of Theorem~2 in~\cite{drton} is proving the lower
bound on the dimension of the classical secant variety. Our computation of the
dimension of the tropical secant variety provides an alternative proof of this
lower bound.
\end{remark}

\section{Tree rank}\label{sec:tree-rank}

The tropical Grasmmannian $G_{2,n}$ is the tropical variety defined by the
$3$-term Pl\"ucker relations:
\begin{equation} \label{eqn:plucker}
x_{ij} x_{k\ell} \oplus x_{ik} x_{j\ell} \oplus x_{i\ell} x_{jk}
\quad\quad\mbox{for all $i < j < k < \ell$}.
\end{equation}
This condition
is equivalent to coming from the distances along a weighted tree which has
negative weights along the internal edges~\cite[Sec.\ 4]{ss}. In this section,
we will always take the deficiency graph to be with respect to the Pl\"ucker
relations in~(\ref{eqn:plucker}).

As with the previous notions of rank, the tree rank of a $0/1$ matrix can be
characterized in terms of covers of graphs. 
For any disjoint subsets $I_1, \ldots, I_k
\subset [n]$ (not necessarily a partition), the \emph{complete $k$-partite
graph} is the graph which has an edge between the elements of $I_i$ and $I_j$
for all $i \neq j$. Complete $ k $-partite graphs are characterized by the
property that among vertices which are incident to some edge, the relation of
having a non-edge is a transitive relation.

\begin{remark}
The complete $k$-partite graphs defined above are exactly those graphs whose
edge set forms the set of bases of a rank~$2$ matroid on $n$ elements.
The transitivity of being a non-edge is equivalent to the basis exchange axiom.
Alternatively, each of the sets $I_1, \ldots, I_k$ partition the set of
non-loops in the matroid into parallel classes. See~\cite{o} for definitions of
these terms. In
the following proposition, we will see that the Pl\"ucker relations imply the
basis exchange axiom for the $0$ entries
of a non-negative tree matrix.
\end{remark}

\begin{prop} \label{prop:tree-rank-01}
Let $M$ be an $n \times n$ $0/1$ dissimilarity matrix and let $r$ be smallest size of a cover
of $G_M$ by complete $k$-partite subgraphs. As in
Proposition~\ref{prop:star-tree-rank-01}, we only
require every edge to be in the cover, not necessarily every vertex.
If $G_M$ has at most one isolated vertex then $M$ has
tree rank $r$. Otherwise, $M$ has tree rank $r+1$.
\end{prop}

\begin{proof}
Let $I_1, \ldots, I_k$ be disjoint sets defining a $k$-partite graph.  We
construct a tree which has $k+1$ internal vertices: one vertex $v_i$ for each $i
\in [k]$ and a vertex
$w$. Every element of $I_i$ has a branch of length $1/2$ to $v_i$ and each $v_i$
has a branch of length $-1/2$ to $w$. The elements of $J = [n] \setminus (I_1
\cup \cdots \cup I_k)$ are connected to $w$
by a branch of length $1$. This gives a distance matrix whose entries are $0$
for the edges of the $k$-partite graph, $2$ for the entries between elements of
$J$, and $1$ elsewhere. If $G_M$ has at most one isolated vertex, then the sum
of these matrices is $M$.  Otherwise, adding the matrix with all entries
entries equal to $1$ yields $M$.

Conversely, suppose we have a decomposition of $M$ as the sum of tree matrices.
For each tree matrix, we can define a graph on the vertices $[n]$ with edges
corresponding to the $0$ entries.
These form a cover of the graph of $M$, so we just need
to show that the graph $G_T$ coming from a tree matrix will be a complete
$k$-partite graph. For this, we need to show that the relation of having a
non-edge is a transitive relation among vertices which are incident to some
edge. Suppose that $(i,j)$ and $(j,k)$ are non-edges, but $(i,k)$ is an edge.
Also suppose that $j$ has an edge to some other vertex $\ell$. Then
$M_{ik} + M_{j\ell} = 0$, but $M_{ij}$ and $M_{jk}$ are positive and
$M_{k\ell}$ and $M_{i\ell}$ are non-negative, which contradicts the Pl\"ucker
relation.  Thus, $G_T$ is a complete $k$-partite graph, and we have proved the
theorem when $G_M$ has at most one isolated vertex.

Now suppose that $G_M$ has at least $2$ isolated vertices, $i$ and~$j$.  There
must be some tree matrix $T$ in the decomposition of $G_M$ such that $T_{ij} =
1$.
Suppose that the corresponding graph $G_T$ has an edge between two vertices $k$
and~$\ell$, which must be distinct from $i$ and~$j$ by assumption. Since $i$
and~$j$ are isolated, $T_{ik}$, $T_{i\ell}$, $T_{jk}$, and~$T_{j\ell}$ must each
be at least~$1$. But $T_{ij} + T_{kl} = 1$, which contradicts the Pl\"ucker
relation. Thus, $G_T$ must be the trivial graph, so the decomposition of $M$
must have one more term than a minimal cover of $G_M$. Therefore, $M$ has rank $r
+ 1$.
\end{proof}

Note that by taking the $I_i$ in the definition of $k$-partite graph to be
singletons, we get complete graphs, and by taking $k=2$ with $I_1$ a singleton
and $I_2$ any set disjoint from $I_1 $, we get star trees. Together with
Propositions~\ref{prop:star-tree-rank-01} and~\ref{prop:tree-rank-01}, this
confirms, for $0/1$-matrices, the second inequality
in~(\ref{eqn:rank-inequals}).

\begin{lemma} \label{lem:extn-tree-rank-1}
Let $M$ be an $m\times m$ tree matrix, $n > m$ an integer, and $C$ any
real number. Then there exists an $n \times n$ tree matrix $N$, whose upper left
$m\times m$ submatrix is $M$ and such that the other entries are each at
least~$C$. 
\end{lemma}

\begin{proof}
The matrix $M$ encodes the distances on some weighted tree $T$ on $m$ leaves.
Pick any internal vertex $v$ of~$T$ and let $C'$ be the smallest distance
between $v$ and a leaf of $m$. Let $T'$ be the tree on $n$ leaves formed from
$T$ by attaching each leaf $i$ with $m < i \leq n$ to $v$ by an edge with weight
$\max\{\frac{1}{2} C, C-C'\}$. Let $N$ be the distance matrix of $T'$ and $N$ is
a tree matrix with the desired properties.
\end{proof}

\begin{prop}
The dimension of the set of dissimilarity $n \times n$ matrices of tree rank at
most $r$ is the dimension of the classical secant variety,
\begin{align*}
{n \choose 2} - {n-2r\choose 2} &\quad\mbox{if } r \leq \frac{n}{2}, \\
{n \choose 2} &\quad\mbox{if } r \geq \frac{n-1}{2}.
\end{align*}
\end{prop}

\begin{proof}
The tropical secant variety is contained in the tropicalization of the classical
variety, which has the given dimension by~\cite[Thm.\ 2.1i]{cgg}. Therefore, it
is sufficient to prove that the tropical secant variety has at least the given
dimension by the Bieri-Groves Theorem~\cite[Thm.\ A]{bg}.

To prove the lower bound on the dimension, first note that for $r = \lfloor
n/2\rfloor$, $n-2r$ is either $0$ or~$1$, so the first part of the statement
implies that the dimension of the $r$th
secant set is  ${n \choose 2}$, the dimension of the space of dissimilarity
$n \times n$ matrices. Since higher secant varieties are at least as large as
the preceeding ones, this implies the second part of the statement.

The proof of the first part, when $r \leq n/2$, is by induction on $n$. For $n
\leq 3$ all dissimilarity matrices have tree rank~$1$, because the four-point
condition is trivial. Thus, the dimension of the $1$st secant set is ${n \choose
2} = {n \choose 2} - {n-2 \choose 2}$, as desired.

Now suppose $n > 3$ and by the inductive hypothesis, let $N$ be an $(n-2) \times
(n-2)$ matrix of tree rank at most $r-1$ such that the locus of matrices
with tree rank at most $r-1$
has dimension ${n-2 \choose 2} - {n - 2r \choose 2}$ in a neighborhood of $N$.

Consider the caterpillar tree with leaves in the order $1$, $3$, $4$, \ldots,
$n$, $2$, pendant edge length $p_i$ for leaf $i$, and negative internal edges
$q_3, \ldots, q_{n-1}$. Thus, the first two rows and columns of the distance
matrix $M$ are defined by:
\begin{align*}
M_{1,2} = M_{2,1} &= p_1 + q_3 + \cdots + q_{n-1} + p_2, \\
M_{1,i} = M_{i,1} &= p_1 + \sum_{j=3}^{i-1} q_j + p_i \quad \mbox{for } i > 2,
\\
M_{2,i} = M_{i,2} & = p_i + \sum_{j=i}^{n-1} q_j + p_2 \quad \mbox{for } i > 2.
\end{align*}
Just from the values in these two rows, we can solve for the values of the edge
lengths:
\begin{align*}
p_1 &= \frac{1}{2}(M_{1,2} + M_{1,3} - M_{2,3}) \\
p_2 &= \frac{1}{2}(M_{1,2} + M_{2,n} - M_{1,n}) \\
p_i &= \frac{1}{2}(M_{1,i} + M_{2,i} - M_{1,2}) \quad\mbox{for } i > 2 \\
q_i &= \frac{1}{2}(M_{1,i+1} + M_{2,i} - M_{1,i} - M_{2,i})
\end{align*}
Therefore, the projection onto the first two columns has dimension $n + (n-3) =
2n-3$ in a neighborhood of this point.

Assume that the $p_i$ are sufficiently negative that the lower right
$(n-2)\times (n-2)$ submatrix of $M$ is less than $N$ in every entry. Then
\begin{equation*}
M \oplus \begin{bmatrix}
* & \infty & \infty  \\
\infty & * & \infty  \\
\infty & \infty & N
\end{bmatrix}
=
\begin{bmatrix}
* & M_{12} & M_{13} & \cdots & M_{1n} \\
M_{12} & * & M_{23} & \cdots & M_{2n} \\
M_{13} & M_{23} & * & \cdots & N_{1,n-2} \\
\vdots & \vdots & \vdots & & \vdots \\
M_{1n} & M_{2n} & N_{n-2,1} & \cdots & *
\end{bmatrix}
\end{equation*}
has tree rank at most $r$ and has a neighborhood of such matrices
of dimension
\begin{equation*}
2n - 3 + {n-2 \choose 2} - {n - 2r  \choose 2} = {n \choose 2} - {n-2r
\choose 2},
\end{equation*}
which is the desired expression.
\end{proof}

Unlike the cases of symmetric Barvinok rank and star tree rank, we do not know
the maximum tree rank of a $n\times n$ dissimilarity matrix for large $n$. We
have an upper bound of $n-2$ by Theorem~\ref{thm:max-star-tree-rank}, and we can
improve on this slightly:

\begin{thm}\label{thm:tree-rank-n-3}
For $n\geq 6$, a $n \times n$ dissimilarity matrix~$M$ has tree rank at most
$n-3$.
\end{thm}

\begin{proof}
Let $M$ be a $6 \times 6$ dissimilarity matrix.
Consider the tropical polynomial whose terms correspond to
the perfect matchings on $6$ vertices:
\begin{gather*}
x_{12}x_{34} x_{56} \oplus x_{12}x_{35}x_{46} \oplus x_{12}x_{36}x_{45}
\oplus x_{13}x_{24}x_{56} \oplus x_{13}x_{25}x_{46} \\
{} \oplus x_{13}x_{26}x_{45}
\oplus x_{14}x_{23}x_{56} \oplus x_{14}x_{25}x_{36} \oplus x_{14}x_{26}x_{35}
\oplus x_{15}x_{23}x_{46} \\
{} \oplus x_{15}x_{24}x_{36} \oplus x_{15}x_{26}x_{34} \oplus x_{16}x_{23}x_{45}
\oplus x_{16}x_{24}x_{35} \oplus x_{16}x_{25}x_{34}.
\end{gather*}
Note that this is the tropicalization of the Pfaffian of a $6\times 6$
dissimilarity matrix (see, for example~\cite[Ch.~7]{godsil}).
After relabeling the vertices, we can assume
that the minimum is achieved by the term $x_{12} x_{34} x_{56}$. In
particular, this means that for the $4$-point
condition applied to the upper left $4\times 4$ matrix, the minimum is achieved
by $x_{12} x_{34}$. We can set $X_1$ equal to the smaller of $M_{13} +
M_{24} - M_{12}$ and $M_{14} + M_{23} - M_{12}$, so that the matrix
\begin{equation*}
N = \begin{bmatrix}
* & M_{12} & M_{13} & M_{14} & \infty & \infty \\
M_{12} & * & M_{23} & M_{24} & \infty & \infty \\
M_{13} & M_{23} & * & X_1 & \infty & \infty \\
M_{14} & M_{24} & X_1 & * & \infty & \infty \\
\infty & \infty & \infty & \infty & * & \infty \\
\infty & \infty & \infty & \infty & \infty & *
\end{bmatrix}
\end{equation*}
is a tree matrix. Moreover, by our assumption, we have that $X_1 \geq M_{34}$.
By similar logic, there exist $X_2$ and $X_3$ such that the
following is an expression of~$M$ as the tropical sum of $3$ tree matrices:
\begin{equation*}
N \oplus
\begin{bmatrix}
* & X_2 & \infty & \infty & M_{15} & M_{16} \\
X_2 & * & \infty & \infty & M_{25} & M_{26} \\
\infty & \infty & * & \infty & \infty & \infty \\
\infty & \infty & \infty & * & \infty & \infty \\
M_{15} & M_{25} & \infty & \infty & * & M_{56} \\
M_{16} & M_{26} & \infty & \infty & M_{56} & *
\end{bmatrix} \\
\oplus 
\begin{bmatrix}
* & \infty & \infty & \infty & \infty & \infty \\
\infty & * & \infty & \infty & \infty & \infty \\
\infty & \infty & * & M_{34} & M_{35} & M_{36} \\
\infty & \infty & M_{34} & * & M_{45} & M_{46} \\
\infty & \infty & M_{35} & M_{45} & * & X_3 \\
\infty & \infty & M_{36} & M_{46} & X_3 & * \\
\end{bmatrix}
\end{equation*}
Therefore, $M$ has tree rank at most~$3$.

For $n > 6$, the theorem follows from Lemma~\ref{lem:tree-rank-submatrix}.
\end{proof}

\begin{table}
\begin{tabular}{|l|l|l|}
\hline
n & maximum tree rank & example \\
\hline
$3$ & $1$ & \\
$4$ & $2$ & \\
$5$ & $3$ & $0/1$ matrix corresponding to 5-cycle \\
$6$ & $3$ & \\
$7$ & $4$ & \\
$8$ & $5$ & \\
$9$ & $6$ & $M$ in~(\ref{eqn:tree-rank-6}) \\
$10$ & $6$ or $7$ & Any extension of $M$ in~(\ref{eqn:tree-rank-6})\\
$9k$ & between $6k$ and $9k-3$ & $M_k$ from discussion
following~(\ref{eqn:tree-rank-6}) \\
\hline
\end{tabular}
\caption{Maximum possible tree rank of an $n\times n$ dissimilarity matrix, to
the best of our knowledge. The
upper bounds come from Theorems~\ref{thm:max-star-tree-rank} 
and~\ref{thm:tree-rank-n-3}. The examples have the largest tree ranks that
are known to us. The omitted examples can be provided by taking a principal
submatrix of a larger example, by Lemma~\ref{lem:tree-rank-submatrix}.}
\label{tbl:tree-rank}
\end{table}
Beginning with $n =10$, we don't know whether or not the bound  in
Theorem~\ref{thm:tree-rank-n-3} is
sharp. For the following $9 \times 9$ matrix, found by random search, the
deficiency graph was computed to have chromatic number $6$:
\begin{equation} \label{eqn:tree-rank-6}
M =  \begin{bmatrix}
* & 1 & 6 & 7 & 2 & 3 & 8 & 9 & 6 \\
1 & * & 2 & 7 & 9 & 7 & 5 & 7 & 1 \\
6 & 2 & * & 6 & 0 & 6 & 1 & 7 & 1 \\
7 & 7 & 6 & * & 3 & 3 & 8 & 5 & 3 \\
2 & 9 & 0 & 3 & * & 5 & 7 & 5 & 7 \\
3 & 7 & 6 & 3 & 5 & * & 9 & 3 & 9 \\
8 & 5 & 1 & 8 & 7 & 9 & * & 2 & 3 \\
9 & 7 & 7 & 5 & 5 & 3 & 2 & * & 8 \\
6 & 1 & 1 & 3 & 7 & 9 & 3 & 8 & * \\
\end{bmatrix}
\end{equation}
Together with Theorem~\ref{thm:tree-rank-n-3}, this computation shows that $M$
has tree rank~$6$.
For any $k \geq 1$,
we can form an $9k\times 9k$ matrix $M_k$ by putting $M$ in blocks along the
diagonal and setting all other entries to $10$. The deficiency graph of
$M_k$ includes $k$ copies of the deficiency graph of $M$, and all edges between
distinct copies. Therefore, the chromatic number, and thus the tree rank, are at
least $6k$.

On the other hand, in order to provide examples of an $n \times n$ matrix with
tree rank $n-3$ for all $n \leq 9$, we have the following lemma.
\begin{lemma} \label{lem:tree-rank-submatrix}
Let $M$ be an $n\times n$ matrix. If any $(n-m) \times (n-m)$ principal
submatrix has tree rank $r$, then $M$ has tree rank at most $r +m$.
\end{lemma}
\begin{proof}
Fix a decomposition of the $(n-m) \times (n-m)$ principal submatrix into $r$
tree matrices. We can extend each tree matrix to an $n\times n$ tree matrix by
Lemma~\ref{lem:extn-tree-rank-1}. For each index~$i$ not in the principal
submatrix, define $v_i$ to be the vector which is $C+M_{ij}$ in the $j$th entry
and $-C$ in the
$i$th entry, where $C$ is a large real number. Then, the extended tree matrices,
together with $\pi(v_i^T \odot v_i)$ for all $i$ not in the principal
submatrix, give a decomposition of $M$ into $r+m$ tree matrices, as
desired.
\end{proof}

These results on the maximum tree rank are summarized in
Table~\ref{tbl:tree-rank}.


\section{Symmetric Barvinok rank for \texorpdfstring{$n = 3$}{n=3}}
\label{sec:sym-barv-3}

In this section, we explicitly describe the stratification of 
$3\times 3$ symmetric matrices by symmetric Barvinok rank. By
Theorem~\ref{thm:max-rank}, the rank is at most~$3$, and the locus of rank~$1$
is the tropical variety defined by the $2\times 2$ minors, so it suffices to
characterize the matrices of rank at most~$2$.

Following~\cite{dss}, we call a square matrix \emph{tropically singular} if it
lies in the tropical variety of the determinant.
\begin{prop} \label{prop:sym-3}
Let $M$ be a symmetric $3\times 3$ matrix. Then the following are equivalent:
\begin{enumerate}
\item $M$ has symmetric Barvinok rank at most $2$;
\item The deficiency graph of $M$ is $2$-colorable;
\item $M$ is tropically singular and 
$M_{ii} + M_{jj} \leq 2 M_{ij}$ for all $1 \leq i, j \leq 3$.
\end{enumerate}
\end{prop}

\begin{proof}
First we note that $M_{ii} + M_{jj} \leq 2M_{ij}$ implies that every term of the
tropical determinant is greater than or equal to $M_{11} + M_{22} + M_{33}$.
Thus, a matrix~$M$ satisfying these inequalities
is tropically singular if and only if some other term of the tropical
determinant equals $M_{11} + M_{22} + M_{33}$.

Proposition~\ref{prop:deficiency-hypergraph} shows that (1) implies~(2).

If the deficiency
graph is $2$-colorable, then it can't have any loops, so $M_{ii} + M_{jj} \leq
2 M_{ij}$.  Furthermore, the three diagonal entries can't form a clique, so
without loss of generality, we assume that there is no edge between $11$ and
$22$.  Together with the inequality, this implies that $M_{11} + M_{22}
= 2M_{12}$, so $M_{11}+M_{22} + M_{33} = 2M_{12} + M_{33}$, so by our initial
remark, $M$ is tropically singular. Therefore (2) implies~(3).

Finally, suppose that $M$ satisfies~(3). We
can subtract $M_{ii}/2$ from the $i$th row and $i$th column without changing
the rank, and so we assume that every diagonal entry is~$0$.  The inequalities
then say that all of the off-diagonal entries are non-negative.  For the minimum
in the tropical determinant to be achieved at least twice, we must have at least
one off-diagonal entry equal to~$0$.  Without loss of generality, we assume that
$M_{12} = 0$. Then,
\begin{equation*}
M= 
\begin{bmatrix}
0 & 0 & M_{13} \\
0 & 0 & M_{23} \\
M_{13} & M_{23} & 0
\end{bmatrix} =
\begin{bmatrix}
0 & 0 & \infty \\
0 & 0 & \infty \\
\infty & \infty & \infty
\end{bmatrix}
\oplus
\begin{bmatrix}
2M_{13} & M_{13} + M_{23} & M_{13} \\
M_{13} + M_{23} & 2 M_{23} & M_{23} \\
M_{13} & M_{23} & 0
\end{bmatrix},
\end{equation*}
where $\infty$ is as in Remark~\ref{rmk:rank-1-extn}.
Therefore, $M$ has symmetric Barvinok rank at most~$2$.
\end{proof}

\begin{remark}
For larger matrices the symmetric Barvinok rank does not have as simple a
characterization as the third condition in Proposition~\ref{prop:sym-3}.
A necessary condition for a symmetric $n\times n$ matrix to have rank at
most~$r$ is that $M_{ii} + M_{jj} \leq 2M_{ij}$ and all the $(r+1) \times (r+1)$
submatrices are tropically singular, but this condition is not sufficient.
For $n \geq 5$ and $r = n$, 
there are $n\times n$ symmetric matrices with finite rank greater
than $n$ by
Remark~\ref{rmk:bipartite_graph}. Even for $n=4$ and for $r=2$
and $r=3$, the matrix
\begin{equation*}
M =
\begin{bmatrix}
0 & 0 & 1 & 2 \\
0 & 0 & 2 & 1 \\
1 & 2 & 0 & 0 \\
2 & 1 & 0 & 0
\end{bmatrix},
\end{equation*}
has symmetric Barvinok rank~$4$, as the nodes $12$, $13$, $24$, and~$34$ form a
$4$-clique in its deficiency graph. However, $M$ and all of its $3\times 3$
submatrices are tropically singular.
\end{remark}

\section{Star tree rank for \texorpdfstring{$n=5$}{n=5}}\label{sec:star-tree-5}

In this section, we give an explicit characterization of the secant set of the
space of star trees in the case $n=5$. We do the same for the Grassmannian in
the next section.

From Theorem~\ref{thm:max-star-tree-rank}, we know that the maximum star tree
rank of a $5 \times 5$ matrix is $3$.  On the other hand, the set of
dissimilarity matrices of star tree
rank~$1$ are defined by the $2\times 2$ minors. Thus, our task is to describe
the second secant set of the space of star trees,
i.e.\ the set of dissimilarity matrices of star tree rank~$2$.

First, we recall the defining ideal of the classical secant variety. The space
of star trees is the tropicalization of the projection of the rank~$1$ symmetric
matrices onto their off-diagonal entries. Its second secant variety is a
hypersurface in $\CC^{10}$ defined by the following $12$-term quintic 
polynomial, known as the pentad~\cite{drton}:
\begin{gather*}
x_{12}x_{13}x_{24}x_{35}x_{45}
-x_{12}x_{13}x_{25}x_{34}x_{45}
-x_{12}x_{14}x_{23}x_{35}x_{45}
{}+x_{12}x_{14}x_{25}x_{34}x_{35} \\
+x_{12}x_{15}x_{23}x_{34}x_{45}
-x_{12}x_{15}x_{24}x_{34}x_{35}
+x_{13}x_{14}x_{23}x_{25}x_{45}
-x_{13}x_{14}x_{24}x_{25}x_{35} \\
{}-x_{13}x_{15}x_{23}x_{24}x_{45}
+x_{13}x_{15}x_{24}x_{25}x_{34}
-x_{14}x_{15}x_{23}x_{25}x_{34}
+x_{14}x_{15}x_{23}x_{24}x_{35}
\end{gather*}
Note that the $12$ terms of the pentad correspond to the $12$
different cycles on 5 vertices.  The second secant set of the space of star
trees is contained in the tropicalization of the pentad, but the containment is
proper. Nonetheless, the terms of the pentad play a fundamental role in
characterizing matrices of rank at most~$2$.

\begin{thm}
Let $M$ be a $5\times 5$ dissimilarity matrix. The following are equivalent:
\begin{enumerate}
\item $M$ has star tree rank at most~$2$;
\item The deficiency graph of $M$ is $2$-colorable;
\item The minimum of the terms of the pentad is achieved at two terms which
satisfy the following conditions:
\begin{enumerate}
\item The terms differ by a transposition;
\item Assuming, without loss of generality, that
the minimized terms are 
$x_{12}x_{23}x_{34}x_{45}x_{15}$
and
$ x_{13}x_{23}x_{24}x_{45}x_{15}$,
then we have that $M_{14}+M_{23} \leq M_{12}+M_{34} =
M_{13}+M_{24}$.
\end{enumerate}
\end{enumerate}
\end{thm}

\begin{proof}
(1) implies (2) by Proposition~\ref{prop:deficiency-hypergraph}.

We show that (2) implies (3) by proving the contrapositive.  Suppose that $M$
doesn't satisfy the two conditions for any pair of terms. Without loss of
generality, we assume that $x_{12}x_{23}x_{34}x_{45}x_{15}$ is a minimal term
from the pentad. By minimality, $M_{12} + M_{34}$ is less than or equal to
$M_{13} + M_{24}$.  If this inequality is strict, then we have an edge between
$12$ and~$34$ in the deficiency graph.  On the other hand, if it is an equality,
then by our assumption that the conditions in~(3) don't hold, $M_{12} + M_{34}$
must be less than $M_{14} + M_{23}$, in which case we also have an edge between
$12$ and~$34$.  Similarly, we have edges between $34$ and~$15$, between $15$
and~$23$, between $23$ and~$45$, and between $45$ and~$12$. Thus, the graph has
a $5$-cycle, and so is not $2$-colorable.  Therefore, $M$ has star tree
rank~$3$.

Finally, suppose that the two conditions in~(3) hold.
Let $A = M_{12}+M_{34} =
M_{13}+M_{24}$. Then we claim that:
\begin{multline*}
M =
\begin{bmatrix}
*  & M_{12} & M_{13} & A - M_{23} & \infty \\
M_{12} & * & M_{23} & M_{24} & \infty \\
M_{13} & M_{23} & * & M_{34} & \infty \\
A - M_{23} & M_{24} & M_{34} & * & \infty \\
\infty & \infty & \infty & \infty & *
\end{bmatrix}
\oplus \\
\def\+{\!\!+\!\!}\def\-{\!\!-\!\!}
\begin{bmatrix}
* & M_{14} \+ M_{25} \- M_{45} & M_{14} \+ M_{35} \- M_{45} & M_{14} & M_{15} \\
M_{14} \+ M_{25} \- M_{45} & * & B & M_{14} \+ M_{25} \- M_{15} & M_{25} \\
M_{14} \+ M_{35} \- M_{45} & B & * & M_{14} \+ M_{35} \- M_{15} & M_{35} \\
M_{14} & M_{14} \+ M_{25} \- M_{15} & M_{14} \+ M_{35} \- M_{15} & * & M_{45} \\
M_{15} & M_{25} & M_{35} & M_{45} & *
\end{bmatrix},
\end{multline*}
where $\infty$ is as in Lemma~\ref{lem:star-tree-extn} and $B
= M_{14} + M_{25} + M_{35} - M_{15} - M_{45}$. Note that each matrix
has some entries taken from~$M$, while the rest are forced by the rank~$1$
condition.
We just need to check
that the minimum of these two matrices is in fact~$M$. To see this, we have the
following inequalities from condition~(3):
\begin{align*}
M_{14}+M_{23} \leq M_{12}+M_{34} &\quad\Rightarrow\quad M_{14} \leq A - M_{23}
\\
M_{12}+M_{23}+M_{34}+M_{45}+M_{15} &\,\leq\,
M_{12}+M_{23}+M_{35}+M_{45}+M_{14} \\
&\quad\Rightarrow\quad M_{34} \leq M_{14} + M_{35} - M_{15} \\
M_{13}+M_{23}+M_{24}+M_{45}+M_{15} &\,\leq\, M_{13}+M_{23}+M_{25}+M_{45}+M_{14}
\\
&\quad\Rightarrow\quad M_{24} \leq M_{14} + M_{25} - M_{15} \\
M_{12} + M_{23} + M_{34} + M_{45} + M_{15} &\,\leq\,
M_{14} + M_{34} + M_{23} + M_{25} + M_{15} \\
&\quad\Rightarrow\quad M_{12}  \leq M_{14} + M_{25} - M_{45} \\
M_{13} + M_{23} + M_{24}  + M_{45} + M_{15} &\,\leq\,
M_{14} + M_{24} + M_{23} + M_{35} + M_{15} \\
&\quad\Rightarrow\quad M_{13} \leq M_{14} + M_{35} - M_{45} \\
M_{12}+M_{23}+M_{34}+M_{45}+M_{15} &\,\leq\,
M_{12}+M_{25}+M_{35}+M_{34}+M_{14} \\
&\quad\Rightarrow\quad M_{23} \leq B
\end{align*}
Therefore, $M$ has star tree rank at most~$2$.
\end{proof}

\section{Tree rank for \texorpdfstring{$n=5$}{n=5}}\label{sec:tree-5}

We now turn our attention to tree rank of $5\times 5$ dissimilarity matrices.
As in the previous section,
the maximum tree rank is~$3$ and so it suffices to characterize $5 \times
5$ dissimilarity matrices of tree rank at most~$2$. Unlike the previous
section, the second classical secant variety is already all of $\mathbb C^{10}$,
so there is no classical polynomial whose tropicalization gives us a clue to the
tropical secant set. However, the tropical pentad again shows up in our
characterization.

First, here is a simple example of a $5\times 5$
$0/1$ dissimilarity matrix with tree rank $3$.
Consider the $0/1$ matrix corresponding to the $5$-cycle $C_5$. Now, $C_5$
cannot be covered by fewer than $3$ $k$-partite graphs, and so the matrix has
tree rank at least $3$ by Proposition~\ref{prop:tree-rank-01}. On the other
hand, it has tree rank at most $3$ by Theorem~\ref{thm:max-star-tree-rank} and
the inequality in~(\ref{eqn:rank-inequals}).
We will see in Remark~\ref{remark_5_cycle}
that this matrix is, in a certain sense, the only such example.

Let $P $ be the tropical polynomial in variables $\{x_{ij}:1\le i< j\le 5\} $
which is the tropical sum of the 22 tropical monomials of degree 5 in which each
$i \in\{1, \ldots ,5\}$  appears in a subscript exactly twice. Thus  $P $ has
$12$ monomials of the form $x_{12}x_{23}x_{34}x_{45}x_{15}$, forming the terms
of the pentad, and $10$ new monomials of the form $x_{12}x_{23}x_{31}x_{45} ^
2$. Let us call terms of the former kind \emph{pentagons}, and terms of the
latter kind \emph{triangles}.

\begin{theorem} \label{thm:tree-rank-5}
Let  $M $ be a $5\times 5 $ dissimilarity matrix. Then the following are
equivalent:
\begin{enumerate}
	\item $M $ has tree rank at most $2$; 
	\item  The deficiency graph is $2$-colorable;
    \item The tropical polynomial~$P$ achieves its minimum at a triangle.
\end{enumerate}
\end{theorem}

\begin{proof}
First, (1) implies (2) by
Proposition~\ref{prop:deficiency-hypergraph}.

For (2) implies (3), we prove the contrapositive. Suppose the minimal terms of
$P$ are
all pentagons; without loss of generality, we assume that 
$x_{12}x_{23}x_{34}x_{45}x_{15}$ is a minimal term. Since
$x_{14}x_{45}x_{15}x_{23}^2$ is not minimal, we have $M_{12} + M_{34} < M_{14} +
M_{23}$. Similarly, we have,
\begin{align*}
M_{12} + M_{23} + M_{34} + M_{45} + M_{15} &< 2M_{15} + M_{23} + M_{34} +
M_{24}, \mbox{ and}\\
M_{12} + M_{23} + M_{34} + M_{45} + M_{15} &< 2M_{45} + M_{12} + M_{23} +
M_{13}.
\end{align*}
Adding these together and cancelling, we get $M_{12} + M_{34} < M_{13} +
M_{24}$.
Thus, $12$ and~$34$ are adjacent in the deficiency graph.  By similar reasoning,
we have adjacencies $12-34-15-23-45-12 $ in the deficiency graph, so it has a
five cycle and is not 2-colorable.  

Finally, we prove that (3) implies (1).  Assume without loss of generality that
$x_{34}x_{35}x_{45}x_{12}^ 2  $  is among the terms minimizing $P $. This
implies that $x_{12}x_{34}$, $x_{12}x_{35}$, and $x_{12}x_{45}$  are each
minimal terms in their respective Pl\"ucker equations.
Then we can use Lemmas~\ref{triangle_complement} and~\ref{triangle} below
to obtain a decomposition of $M $  into two tree matrices.
 
\begin{lemma}\label{triangle_complement}
For any  $5 \times 5$ dissimilarity matrix $M$ such that $x_{12}x_{34}$,
$x_{12}x_{35}$, and $x_{12}x_{45}$ are each minimal terms in their respective
Pl\"ucker equations, there exists some $5 \times 5$ tree
matrix $T$, such that for every $ij \in{[5] \choose 2} $, we
have $T_{ij}\ge  M_{ij} $, with equality if $ij\in\{12, 13, 14, 15, 23, 24, 25\}
$. 
\end{lemma}
\begin{proof}
If $ M$  satisfies
\begin{align*}
M_{14}+ M_{23} &\le  M_{13}+ M_{24}, \\
M_{15}+ M_{24} &\le  M_{14}+ M_{25}, \\
M_{13}+ M_{25} &\le M_{15}+ M_{23},
\end{align*}
then adding shows that each inequality must be an equality. Thus, without loss
of generality, we may assume that
\begin{align}
\label{eqn:inequalities_1}M_{14}+ M_{23} &\ge   M_{13}+ M_{24}, \\
\label{eqn:inequalities_2}M_{15}+ M_{24} &\le  M_{14}+ M_{25}, \\\
\label{eqn:inequalities_3}M_{13}+ M_{25} &\le M_{15}+ M_{23}.
\end{align}
Every other case is equivalent to this one via permutations of $\{1,2\}$ and
$\{3,4,5\}$.

Now define $T $  as follows. Let $ T_{ij}= M_{ij} $ for $ij\in\{12, 13, 14, 15,
23, 24, 25\} $, and
\begin{align*}
 T_{34} & = T_{24} + T_{ 13} - T_{12},\\
 T_{35} & = T_{25} + T_{ 13} - T_{12},\\
 T_{45} & = T_{15} + T_{ 24} - T_{12}.
\end{align*} 
That $T $  dominates $ M$  in every coordinate follows from the
inequalities~(\ref{eqn:inequalities_1}), (\ref{eqn:inequalities_2}),
and~(\ref{eqn:inequalities_3}).  To check that $T $  has tree rank 1, it
suffices to check the four-point condition on each 4-tuple. This condition is
satisfied for the 4-tuples $\{1,2,3,4\} $, $ \{1,2,3,5\}$, and $ \{1,2,4,5\}$,
by choice of $T_{34},T_{35},T_{45} $. Furthermore, we claim that
\begin{align*}
T_{15} +T_{34} &=T_{13} +T_{45} \le  T_{14} +T_{35},\textrm{ and} \\
T_{24} +T_{35} & =T_{25} +T_{34} \le  T_{23} +T_{45}.
\end{align*}
These inequalities follow immediately from substituting for
$T_{34},T_{35},T_{45} $  and using inequalities (\ref{eqn:inequalities_2})
and~(\ref{eqn:inequalities_3}).
\end{proof}

\begin{lemma}\label{triangle}
For any  $5 \times 5$ dissimilarity matrix $M$, there exists some $5 \times 5$
tree matrix $T'$ such that for every pair of indices $i$ and $j$, we have
$T'_{ij}\ge  M_{ij} $, with equality if $ij\in\{34, 35, 45\} $. 
\end{lemma}
\begin{proof}
The $\{3,4,5\}$-principal submatrix of $M$ is a tree matrix and therefore $T'$
exists with the desired properties by Lemma~\ref{lem:extn-tree-rank-1}.
\end{proof}

We can now finish the proof of Theorem~\ref{thm:tree-rank-5}.  Let $T $  be as
given in Lemma~\ref{triangle_complement} and let $T'$  be as
given in Lemma~\ref{triangle}. Then
$ M = T \oplus  T'$
and so $ M$  has tree rank at most~2.
\end{proof}

In the rest of this section, we present a detailed analysis of the
deficiency graphs~$\Delta_M$ for $n=5$. There are $5$ tropical
Pl\"ucker relations on a $5 \times 5$ matrix, each containing $3$ terms. Each
term is the tropical product of terms with disjoint entries. Thus, $\Delta_M$ is
a subgraph with at most $5$ edges of the Petersen graph~$P_{10}$, which 
is the graph on vertices ${[5] \choose 2}$ with an edge between $ij$ and $kl$ if and only if $\{i, j\}$
and $\{k, l\}$ are disjoint sets. The following theorem describes the possible
subgraphs that $\Delta_M$ can be.

\begin{figure}
\begin{center}
\input{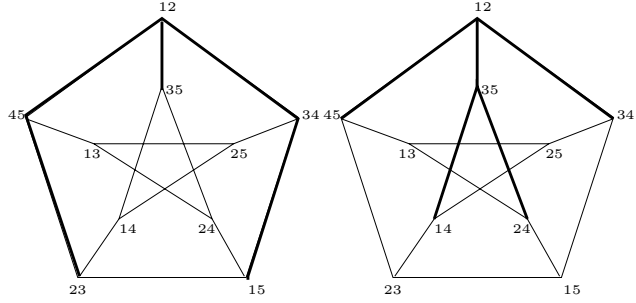}
\caption{The two 2-colorable possibilities for $\Delta_M$.}
\label{fig:Petersen-subgraphs}
\end{center}
\end{figure}
\begin{theorem}\label{thm:tree-rank-5-deficiency}
Let $M$ be a $5\times 5$ dissimilarity matrix.
Then the deficiency graph $\Delta_M$ is precisely one of the following:
\begin{enumerate}
\item The trivial graph, in which case $M$ has tree rank~$1$.
\item A non-trivial graph with fewer than $5$ edges, in which case $M$ has tree
rank~$2$.
\item Up to relabeling, either of the two graphs in
Figure~\ref{fig:Petersen-subgraphs}, in which case $M$ has tree rank~$2$.
\item A $5$-cycle, in which case $M$ has tree rank~$3$.
\end{enumerate}
\end{theorem}
\begin{proof}
The matrix $M$ is a tree matrix if and only if the four-point condition holds
for all $4$-tuples, i.e.\ if and only if $\Delta_M$ is trivial. This is the
first case.

Now suppose that $\Delta_M$ is a non-trivial graph with at most $4$ edges.
Then, at least one four-point condition holds, so
Lemma~\ref{lem:tree-rank-submatrix} implies that $M$ has tree rank at most~$2$.
However, at least one four-point condition is violated, so $M$ must have tree
rank exactly~$2$.

In the rest of the proof, we assume that $\Delta_M$ has $5$ edges, i.e.\ that
each of the five 4-tuples yields one deficient pair. Thus, $\Delta_M$ contains
exactly one of the three edges $12-34$, $13-24$, and $14-23$,  and
similarly for the remaining 4-tuples. 
\begin{figure}
\begin{center}
\input{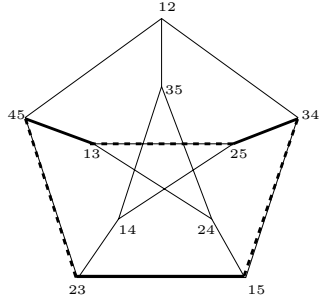}
\caption{An alternating 6-cycle, where the solid edges lie in $H $  and the dotted edges do not.}
\label{fig:alternating-cycle-1}
\end{center}
\end{figure}
Let us say that a subgraph $H $ of  $P_{10} $ \emph{admits an alternating even
cycle} if there exists a cycle of even length  in $ P_{10}$  such that,
traversing the cycle, the edges are alternately members and nonmembers of $ H$
(Figure~\ref{fig:alternating-cycle-1}).
\begin{lemma}\label{lem:alternating-cycle}
The graph $ \Delta_M$ admits no alternating even cycle.
\end{lemma}
\begin{proof}
Note that the only even cycles in the Petersen graph are of lengths 6 and~8
(see, for example, \cite[Ch.~4.2]{bm}).
Suppose  $ C$ is an alternating cycle of length~6. All 6-cycles of $P_{10}$ are
isomorphic, so after relabeling, we may assume that  $ C$ has vertices $ 45,
13,25,34, 15,23$, in that order. Now suppose that $45-13$, $25-34$, and $15-23$
are the three edges of the cycle in $H$. Then
\begin{align*}
M_{45}+ M_{13} &< M_{34}+ M_{15}, \\
M_{25}+ M_{34} &< M_{23}+ M_{45}, \\
M_{15}+ M_{23} &< M_{13}+ M_{25}.
\end{align*}
But adding these inequalities yields a contradiction. If instead $13-
25$, $34-15$, and $23-45$  are the three edges of the cycle lying in $
H$, then we obtain a similar contradiction.

The case of a cycle of length 8 is analogous, so we omit it.
\end{proof}

We now use Lemma~\ref{lem:alternating-cycle} to show that, up to relabeling, the
only two possibilities for $\Delta_M$, assuming that it is $2$-colorable, are
those in Figure~\ref{fig:Petersen-subgraphs}.  We
organize our case analysis according to the maximum degree in the graph
$\Delta_M$.

If $\Delta_M$ has maximum degree~1, so that it is a perfect matching in
$P_{10}$, then one may check that it has an alternating 8-cycle. This is
impossible by Lemma~\ref{lem:alternating-cycle}.
 
If $\Delta_M$ has maximum degree~2, then it is either a
4-path and a 1-path, a 3-path and two 1-paths, or two 2-paths and a 1-path. (By
a $ k $-path we mean a path with $k $~edges). It is easy to check that in each
of these cases, $\Delta_M$ admits an alternating $6 $-cycle, which is again
impossible by Lemma~\ref{lem:alternating-cycle}.
 
Thus, $\Delta_M$ has a vertex of degree 3, which we assume to be the
vertex~$12$. In this case, one may check that, up to relabeling, $\Delta_M$ is
one of the two
graphs shown in Figure~\ref{fig:Petersen-subgraphs}.  
Note that in either case, the graphs are $2$-colorable, and thus $M$ has tree
rank~$2$ by Theorem~\ref{thm:tree-rank-5}.

Finally, if $\Delta_M$ is not $2$-colorable, then it must have an odd cycle. The
Petersen graph has no $3$-cycles, so $\Delta_M$ must be a $5$-cycle.
\end{proof}

\begin{remark}\label{remark_5_cycle}
If $M$ is the $0/1$ matrix corresponding to the $5$-cycle $C_5$, then
$\Delta_M$ is also a $5$-cycle by Theorem~\ref{thm:tree-rank-5-deficiency}.
Explicitly, $\Delta_M$ has an edge for each non-adjacent pair of edges in the
graph $C_5$. Moreover, Theorem~\ref{thm:tree-rank-5-deficiency} tells us that
any other matrix $N$ with tree rank~$3$ must have the same deficiency graph (up
to relabeling). In this sense, $M$ is essentially the only example of a
$5\times 5$ dissimilarity matrix with tree rank~$3$.
\end{remark}

\section{Open questions}\label{sec:open-questions}
\begin{itemize}
\item What is the maximum tree rank of a $n\times n$ dissimilarity matrix?
Theorem~\ref{thm:tree-rank-n-3} gives an upper bound, but beginning with
$n=10$, we do not know if it is sharp.
Specifically, does there exist a $10 \times 10$ dissimilarity matrix with tree
rank $7$? 
\item Give a (reasonable) algorithm for computing tree rank. Note that both the
tropical Veronese and the space of star trees are classical linear spaces, and
so the results in~\cite{de} can be applied to compute
star tree rank and symmetric Barvinok rank. However, it would be nice to have a
good algorithm for computing tree rank.
\item Is it true that the rank of a matrix, according to any of our notions, is
always equal to the chromatic number of the corresponding deficiency graph?
In Sections~\ref{sec:sym-barv-3}, \ref{sec:star-tree-5}, and~\ref{sec:tree-5},
we observed that this was true
for symmetric Barvinok rank with $n \leq 3$, star tree rank with $n \leq 5$, and
tree rank with $n \leq 5$ respectively. In general, we believe the answer is no,
but we do not know of a counterexample.
\item In phylogenetics, trees have all edges, including the pendant
edges, labeled by positive weights, or, equivalently, after negating, by
negative weights. In this way, phylogenetic trees form a subset of the
tropical Grassmannian, and we define the \emph{phylogenetic tree rank} to be the
rank with respect to this subset. One can ask the same questions about
phylogenetic tree rank as in this paper:
Which matrices have finite phylogenetic tree rank? What is the maximum
possible finite phylogenetic tree rank? What is an explicit characterization of
the secant sets for small matrices? Work in this direction was begun
in~\cite{c}.
\item What is an algorithm for computing phylogenetic tree rank?
\item Does a matrix have tree rank~$2$ if and only if all its principal $6
\times 6$ submatrices have tree rank at most~$2$? Pachter and Sturmfels have
made the same conjecture for their definition of tree rank~\cite[p.~ 124]{ps}.
\item What is an explicit description of $6\times 6$ matrices with tree rank at
most~$2$, along the lines of Theorem~\ref{thm:tree-rank-5}? Does this help with
the previous conjecture? There is a necessary condition coming from applying
Theorem~\ref{thm:tree-rank-5} to every $5 \times 5$ minor, and another from the
fact that the matrix must be in the tropicalization of the
classical secant variety.  However, these two conditions
together may not be sufficient for a $6\times 6$ matrix to have tree rank at
most~$2$.
\end{itemize}

\section*{Acknowledgments}
We thank Bernd Sturmfels for his guidance and a close reading of the text. Maria
Ang\'elica Cueto helped us understand the connection to phylogenetics and
William Slofstra pointed us to~\cite{godsil}.
Melody Chan was supported by the Department of Defense through the National
Defense Science and Engineering Graduate Fellowship Program.

\end{document}